\documentclass[10pt]{amsart}
\usepackage[utf8]{inputenc}

\usepackage{amssymb,amsthm,amsmath}
\usepackage{mathrsfs}
\usepackage{enumerate}
\usepackage{graphicx,xcolor}
\usepackage{setspace}
\usepackage[hidelinks]{hyperref}
\usepackage{comment}
\usepackage{bbm}
\usepackage{pgfplots}
\usepgfplotslibrary{fillbetween}
\usetikzlibrary{calc}
\pgfplotsset{compat=1.18}
\usepackage[paper=a4paper, left=1.2in, right=1.2in, top=1.2in, bottom=1.2in]{geometry}
\newcommand{\ls}{\leqslant}
\newcommand{\gr}{\geqslant}
\newcommand{\E}{\mathbb{E}}

\newcommand{\R}{\mathbb{R}}

\DeclareMathOperator{\vol}{vol}

\usepackage{dsfont}

\numberwithin{equation}{section} % Equations numbered by section

\newtheorem{theorem}{Theorem}[section] % <-- Add [section] here
\newtheorem{lemma}[theorem]{Lemma}

\newtheorem{proposition}[theorem]{Proposition}

\theoremstyle{remark}
\newtheorem{remark}[theorem]{Remark}
 
\newtheorem{conjecture}[theorem]{Conjecture}

\theoremstyle{definition}

\hypersetup{
    colorlinks,
    linkcolor={teal!50!black},
    citecolor={green!40!black},
    urlcolor={cyan!60!black}
}

\title{On Extremal Volume Projections of the Simplex and the Cube}

\author{Christos Pandis}

\address{Department of Mathematics \& Applied Mathematics, University of Crete, Voutes Campus, 70013 Heraklion, Greece}

\email{chrpandis@gmail.com}

\date{\today}

\begin{document}

\maketitle
\begin{abstract}
Let $\Delta_n$ and $Q_n$ denote the regular $n$-simplex of side length $\sqrt{2}$ embedded in $\mathbb{R}^{n+1}$ and the volume one cube in $\mathbb{R}^n$, respectively. We derive a closed-form formula for the hyperplane volume projections of $\Delta_n$, which also yields the directions achieving the extremal volume. Moreover, we revisit the problem of extremal planar projections of $Q_n$. In addition, we present generalizations within the framework of $L_p$-projection bodies.
\end{abstract}

\section{Introduction and Results}
The study of lower-dimensional sections and projections of convex bodies is a central topic in asymptotic geometric analysis and high-dimensional geometry.  This has motivated the development and interplay of a wide range of analytic, geometric, and probabilistic techniques.  Interestingly, the nature of the methods often differs depending on whether one studies sections or projections: while problems involving sections tend to rely more heavily on analytic tools, projection problems often involve more algebraic and combinatorial approaches. Given the vast literature on adjacent topics, we refer the interested reader to the recent comprehensive survey \cite{nayar2023extremal}, as well as to \cite{Zhong} for the case of cubes specifically, for a detailed historical account, the current state of the art, and numerous intriguing open problems.
Throughout this work we denote $\Delta_n$  as the  regular $n$-dimensional simplex of side length $\sqrt{2}$ embedded in $\mathbb{R}^{n+1}$, $B_1^n=\text{conv}\{ \pm e_1,\ldots, \pm e_n\}$ the cross-polytope in $\mathbb{R}^n$ and the cube $B_{\infty}^n=[-1,1]^n$. The cube often warrants the volume 1 normalisation as~$Q_n=\left[-\frac{1}{2},\frac{1}{2} \right]^n$.

\subsection{Extremal directions for hyperplane projections of $\Delta_n$ }
Let $\Delta_n$ denote the regular $n$-dimensional simplex of side length $\sqrt{2}$, which we view through its usual embedding in $\mathcal{H}=\left \{\sum_{i=1}^{n+1}x_i=1 \right\}$ in $\mathbb{R}^{n+1}$, namely
\[
\Delta_n = \left\{x\in\mathbb{R}^{n+1}: x_j\gr 0, \sum_{j=1}^{n+1}x_j=1\right\}.
\] 
The problem of interest can be described as follows: Find
\[
\max_{H}  /  \min_{H}\, \operatorname{vol}_{n-1}\left(\operatorname{Proj}_{H}\Delta_n \right),
\]
where the maximum (respectively, minimum) is taken over all affine subspaces 
$H$ of $\mathcal{H}$ of relative codimension one. Since volume of projections is invariant under translations we may assume that they pass  through the barycenter $\text{bar}(\Delta_n)$. Every such hyperplane extends to a codimension $1$ subspace of $\mathbb{R}^{n+1}$ by taking the affine hull of $H$ and the origin, yielding a hyperplane $a^{\perp}$ with $a$ lying in $\mathcal{H}'=\left\{\sum_{i=1}^{n+1}x_i=0 \right\}$. The main difference with the respective slicing problem (see \cite{webb1996central}) is that the barycenter condition is not imposed a priori but arises naturally. Thus, the problem reformulates as follows:
\[
\max_{a\in \mathcal{H}'\cap S^{n}}  /  \min_{a\in \mathcal{H}'\cap S^{n}}\, \operatorname{vol}_{n-1}\left(\operatorname{Proj}_{a^{\perp}\cap \mathcal{H}}\Delta_n \right).
\]

Martini and Weissbach showed in \cite{weissbach1984besten} that the hyperplane containing the smallest projection of $\Delta_n$  is normal to an edge of $\Delta_n$  (see \cite[pp. 164-166]{weissbach1984besten} while largest hyperplane projection is the
transform of a 1-simplex with even weights (see \cite[p. 165]{weissbach1984besten} and \cite{filliman1990extreme}). In other words the largest $(n-1)$-dimensional volume of $\Delta_n$  is obtained by projecting in the direction of the sum of any $\frac{n+1}{2}$  of the normals to its facets, if $n$ is odd and the sum of any of $\frac{n}{2}$ or $\frac{n+2}{2}$ of the normals to the facets if $n$ is even. For example, the largest $3$-dimensional projection of  $\Delta_4$ is obtained  by projecting in the direction $a=\frac{1}{\sqrt{30}}\left(-3,-3,2,2,2 \right)$ and the projection lies in the $3$-plane in $\R^5$ orthogonal to $\text{span}\{a,(1,1,1,1,1)\}$. 

In a different approach, in \cite{martini1992quermasses} Martini and Weissbach established that for an arbitrary $d$-simplex $S$ ($d\gr2$) and for every  direction $a\in\mathbb{S}^{d-1}$ 
\begin{equation}\label{width-proj}
\frac{\operatorname{vol}_{d-1}\left(\text{Proj}_{a^{\perp}}S \right)}{g(DS,a)}=d\operatorname{vol}_d(S),
\end{equation}
where $g(K,a):=\min \{\rho>0:\, a\in \rho K \}$ is the distance function and $DK:=K-K$ is the difference body. They relied on the fact that extremal values of $\frac{1}{g(DK,a)}$ and $b(K,a)$ coincide, where $b(K,a):=h_{K}(a)+h_{K}(-a)$ gives the width of $K$ in direction $a$, that is $\max b(K,a)=\max \frac{1}{g(DK,a)}$ and $\min b(K,a)=\min \frac{1}{g(DK,a)}$.  Thus, establishing bounds using the width $b(S,a)$. For a general characterization of \eqref{width-proj} see \cite{martini1991convex} and concerning the extremal width of the simplex see \cite{alexander1977width,gritzmann1992inner,har2023width,weissbach1988schranken}. 

The projections of the regular simplex onto
certain orthogonal complementary subspaces are conjectured to yield minimal and maximal volume.
\begin{conjecture}[Filliman \cite{filliman1990extreme}, or see \cite{nayar2023extremal}] \label{conj:Filliman}
Fix $1 \le k \le n$. Let $H_{\ast}$ be a k-dimensional subspace in $\mathbb{R}^n$ such that $T_{\ast} = \text{Proj}_{H_{\ast}}(\Delta_n)$ is a $k$-dimensional simplex, with the vertices of $\Delta_n$ projecting only onto the vertices of $T_{\ast}$, as evenly as possible: for each $i \le k + 1$, letting $w_i$ be the number of vertices of $\Delta_n$ projecting onto vertex $i$ of $T_{\ast}$, we have that
\[
w_i = \begin{cases} \ell + 1, & 1 \le i \le r, \\ \ell, & r < i \le k + 1, \end{cases}
\]
where we divide $n + 1$ by $k + 1$ with the remainder $r \in \{0, \ldots, k\}$, $n + 1 = (k + 1)\ell + r$. Then
\[
\min_{\substack{H \subset \mathbb{R}^n \\ \dim H=k}} \text{vol}_k(\text{Proj}_H(\Delta_n)) = \text{vol}_k(T_{\ast}).
\]
Moreover, the polytope $T^{\ast} = \text{Proj}_{H_{\ast}^{\perp}}(\Delta_n)$ is conjectured to maximise the volume of projections onto the $n - k$ dimensional subspaces,
\[
\max_{\substack{H \subset \mathbb{R}^n \\ \dim H=n-k}} \text{vol}_{n-k}(\text{Proj}_H(\Delta_n)) = \text{vol}_{n-k}(T^{\ast}).
\]
\end{conjecture}
Filliman developed exterior algebra techniques in \cite{filliman1992volume} and used them in \cite{filliman1990extreme} (see also \cite{filliman1988largest,filliman1990exterior}) to confirm this conjecture in the following cases, for the minimum: $k = 1, 2, n-1$ and arbitrary $n$, as well as $n \le 6$ and arbitrary $k$, and for the maximum: $k = 1, 2, n - 1$ and arbitrary $n$, $k = n - 2$ and $n \le 8$, as well as $k = 3$ and $n = 6$. 

Closed-form expressions for the hyperplane projections of other polytopes, such as the unit cube \(Q_n\) and the cross-polytope $B_1^n$, are well known (see below as well as \cite{nayar2023extremal}). By contrast, no closed-form formula is known for the regular simplex $\Delta_n$.

In this work we give a direct proof identifying the extremal hyperplane projections of the simplex via a closed-form volume formula and the $\ell_1$-norm of the  direction; this approach also determines the directions that realize the extremal volumes.  Our main theorem states the following:
\begin{theorem}\label{main thm}
Let $a\in \mathbb{S}^{n}\cap \mathcal{H}'$, that is $\sum_{i=1}^{n+1}a_i=0$ and $\sum_{i=1}^{n+1}a_i^2=1$. Then,
\begin{equation}\label{formula proj simp}
    \operatorname{vol}_{n-1}\left(\operatorname{Proj}_{a^{\perp}\cap\mathcal{H}}\Delta_n \right)=\frac{1}{2}\frac{\sqrt{n+1}}{(n-1)!}\sum_{j=1}^{n+1} \left|a_j\right|,
\end{equation}
\end{theorem}

\begin{proposition}\label{sum min}
Let $a_1, \ldots, a_{n}$ be real numbers satisfying
\[
a_1 + \ldots + a_{n} = 0.
\]
Then, the following inequality holds:
\[
\sum_{j=1}^{n} |a_j| \geq \sqrt{2} \left( a_1^2 + \ldots + a_{n}^2 \right)^{1/2}.
\]
\end{proposition}
\begin{proposition}\label{sum max}
Let $a_1, \ldots, a_{n} \in \mathbb{R}$, with $n \gr 2$, be real numbers satisfying
\[
a_1 + \ldots + a_{n} = 0 \quad \text{and} \quad a_1^2 + \ldots + a_{n}^2 = 1.
\]
Then, the quantity
\[
\sum_{j=1}^{n} |a_j|
\]
is maximized at the “half-plus/half-minus” vector
\[
\left( \underbrace{\frac{1}{\sqrt{n}}, \ldots, \frac{1}{\sqrt{n}}}_{\frac{n}{2} \text{ times}}, \underbrace{-\frac{1}{\sqrt{n}}, \ldots, -\frac{1}{\sqrt{n}}}_{\frac{n}{2} \text{ times}} \right),
\]
when $n$ is even, and at the at the $(\tfrac{n}{2}-1, \tfrac{n}{2}+1)$–split vector

$$\left(\underbrace{\sqrt{\frac{n+1}{n(n-1)}}, \ldots, \sqrt{\frac{n+1}{n(n-1)}}}_{\frac{n-1}{2} \text{  times}}, \underbrace{-\sqrt{\frac{n-1}{n(n+1)}}, \ldots, -\sqrt{\frac{n-1}{n(n+1)}}}_{\frac{n+1}{2} \text{ times}}\right)$$
when $n$ is odd.
\end{proposition}

Using Theorem \ref{main thm}, we establish in Propositions \ref{sum min} and \ref{sum max} that the minimum is attained at $$ \left(\frac{1}{\sqrt{2}}, -\frac{1}{\sqrt{2}}, 0, \ldots, 0\right),$$
while the maximum is achieved at  at the “half-plus/half-minus” vector $$ \left(\underbrace{\frac{1}{\sqrt{n+1}}, \dots, \frac{1}{\sqrt{n+1}}}_{\frac{n+1}{2} \text{ times}}, \underbrace{-\frac{1}{\sqrt{n+1}}, \dots, -\frac{1}{\sqrt{n+1}}}_{\frac{n+1}{2}  \text{ times}}\right),$$ when \(n+1\) is even, and at $$\left(\underbrace{\sqrt{\frac{n+2}{n(n+1)}}, \ldots, \sqrt{\frac{n+2}{n(n+1)}}}_{\frac{n}{2} \text{ times}}, \underbrace{-\sqrt{\frac{n}{(n+1)(n+2)}}, \ldots, -\sqrt{\frac{n}{(n+1)(n+2)}}}_{\frac{n+2}{2} \text{ times}}\right)$$
when $n+1$ is odd. A direct computation reveals the extremal values
\[
V_{\text{min}}=\frac{1}{\sqrt{2}}\frac{\sqrt{n+1}}{(n-1)!}
\]
and 
\[
V_{\text{max}}=\frac{1}{2(n-1)!}\begin{cases} n+1, & \text{if } n \text{ is odd} \\ \sqrt{n(n+2)}, & \text{if } n \text{ is even} \end{cases}  \quad.
\]
Since $\Delta_n$  is permutational invariant these directions are determined only up to permutation and also under sign reflections due to \eqref{formula proj simp}, that is a direction $x$ and $-x$ yield the same value.

\begin{remark}
 Moreover, we also address the question of determining the extremal directions for the width of $\Delta_n$. For  $a\in \mathbb{S}^{n}\cap \mathcal{H}'$, that is $\sum_{j=1}^{n+1}a_j=0$ and $\sum_{j=1}^{n+1}a_j^2=1$, one gets 
\begin{equation}\label{g(D D^n,u)}
g(D\Delta_n,a)=\frac{1}{2}\sum_{j=1}^{n+1}|a_j|.
\end{equation}
To see this we observe that 
\begin{equation}\label{D^n-D^n}
\Delta_n-\Delta_n=\left \{x\in \mathbb{R}^{n+1}:\, \sum_{j=1}^{n+1}x_j=0 \,\,\,\text{and}\,\,\, \sum_{j=1}^{n+1}|x_j|\ls2 \right\} \subset \mathcal{H}'.
\end{equation}
For the non-obvious inclusion, notice that if $x\in \mathbb{R}^{n+1}$ such that $\sum_{j=1}^{n+1}x_j=0$ and $\sum_{j=1}^{n+1}|x_j|\ls2$ then $$x_j=x_j^{+}-x_j^{-}=x_j^{+}-\frac{1}{n+1}(s-1)-[x_j^{-}-\frac{1}{n+1}(s-1)],$$ where $x^{+}:=\max\{x,0\}$, $x^{-}:=\max\{-x,0 \}$ and $s:=\sum x_j^{+}=\sum x_j^{-}=\frac{\sum|x_j|}{2}\ls1$. 

Combining \eqref{formula proj simp} and \eqref{g(D D^n,u)} we obtain relation~\eqref{width-proj} for $\Delta_n$.

In fact, relation~\ref{g(D D^n,u)} together with Propositions \ref{sum min} and \ref{sum max}  suggests that the extremal width of $\Delta_n$, since one only need to consider directions that are parallel to $\mathcal{H}$, is attained at
    \[
    \max_{a\in \mathbb{S}^{n}\cap \mathcal{H}'}b_{\Delta_n}(a)=\max_{a\in \mathbb{S}^{n}\cap \mathcal{H}'}\frac{1}{g(D\Delta_n,a)}= \sqrt{2}
    \]
    and 
    \[
    \min_{a\in \mathbb{S}^{n}\cap \mathcal{H}'} b_{\Delta_n}(a)=\min_{a\in \mathbb{S}^{n}\cap \mathcal{H}'}\frac{1}{g(D\Delta_n,a)} =\begin{cases} \frac{2}{\sqrt{n+1}}, & \text{if } n \text{ is odd} \\ 2\sqrt{\frac{n+1}{n(n+2)}}, & \text{if } n \text{ is even} \end{cases} 
    \]
   where the directions attaining them are precisely the opposites of those corresponding to the projections.
\end{remark}

\subsection{Extremal planar projections of $Q_n$}
We begin with a   striking and remarkable feature of the cube: its projections onto orthogonal complementary subspaces have the same volume: 
\begin{theorem}[McMullen~\cite{mcmullen1984volumes}; Chakerian--Filliman~\cite{chakerian1986measures}]\label{symmetry cube}
Let $1\ls k \ls n$. For every $k$-dimensional subspace $H$ in $\mathbb{R}^n$, we have 
\[
\operatorname{vol}_k(\operatorname{Proj}_H Q_n)=\operatorname{vol}_{n-k}(\operatorname{Proj}_{H^{\perp}}Q_n).
\]
\end{theorem}
This result was obtained by McMullen and, independently, by Chakerian and Filliman, using the same approach based on Shephard's formula. In particular, sharp bounds on the volumes of $k$-dimensional projections are equivalent to sharp bounds on the volumes of $(n-k)$-dimensional projections.

Let $a\in \mathbb{S}^{n-1}$, then one easily observes that Theorem~\ref{symmetry cube} suggests
\begin{equation}\label{cauchy cube}
\operatorname{vol}_{n-1}(\operatorname{Proj}_{a^{\perp}}Q_n)=\sum_{j=1}^n|a_j|.
\end{equation}
Therefore,
\[
1\ls  \operatorname{vol}_{n-1}(\operatorname{Proj}_{a^{\perp}}Q_n) \ls \sqrt{n}
\]
The lower bound is attained if and only if $Q_n$ is projected onto a coordinate hyperplane, and the upper bound is attained if and only if $Q_n$ is projected onto a hyperplane orthogonal to a main diagonal.

Moreover, it was proven by Chakerian-Filliman that for every $k$ dimensional subspace $H$ $$\operatorname{vol}_k(\operatorname{Proj}_HQ_n)\gr1,$$ and is sharp for all $k$ attained at coordinate subspaces.

 The upper bound for the planar problem was originally resolved in \cite{chakerian1986measures}, where Chakerian and Filliman, using the isoperimetric inequality for polygons, obtained a sharp upper bound for two-dimensional projections and consequently for $(n-2)$-dimensional ones as well.
 \begin{theorem}[Chakerian-Filliman]\label{planar proj cube Chak-Fill}
 Let $n\gr2$. For every $2$-dimensional subspace $H$ in $\mathbb{R}^n$ we have 
 \[
  \operatorname{vol}_2(\operatorname{Proj}_H Q_n)\ls \cot\left(\frac{\pi}{2n}\right).
 \]
 with equality being attained when $\operatorname{Proj}_HQ_n$ is a regular $2n$-gon of edge length $\sqrt{\frac{2}{n}}$.
 \end{theorem}

Our approach proceeds via the derivation of a closed-form volume formula, accompanied by sharp estimates.

Let $\Pi$ be a plane in $\mathbb{R}^n$. This plane is uniquely determined by two orthonormal vectors $u, v \in \mathbb{S}^{n-1}$, so that $\Pi = \operatorname{span}\{u, v\}$. In this section, we focus on the following problem:
\[
\max_{\Pi} / \min_{\Pi} \, \operatorname{vol}_2(\operatorname{Proj}_{\Pi} Q_n),
\]
which also addresses the $(n-2)$-dimensional case. This problem reduces to determining the vectors $u$ and $v$ satisfying
\[
\sum_{j=1}^n u_j^2 = \sum_{j=1}^n v_j^2 = 1 \quad \text{and} \quad \sum_{j=1}^n u_j v_j = 0
\]
that achieve extremal volumes. As before, we provide a closed-form formula that also yields these extremal values.

\begin{proposition}\label{formula for planar cube}
Let  $u,v\in \mathbb{R}^{n}$ orthogonal, then  
\[
\operatorname{vol}_{n-2}(\operatorname{Proj}_{\text{span}^{\perp}\{u,v\} }Q_n)=\sum_{1\ls i<j\ls n}|u_iv_j-u_j v_i|
\]
\end{proposition}

We conclude by establishing sharp bounds for the latter quantity.
\begin{proposition}\label{bound planar cube}
Let $u,v\in \mathbb{S}^{n-1}$. Then,
\[
 \sum_{1\ls i<j\ls n} |u_iv_j-u_jv_i|\ls \cot\left( \frac{\pi}{2n}\right),
\]
and equality is attained when $(u_i,v_i)=\sqrt{\frac{2}{n}} \left(\cos\left(\frac{ (\sigma(i)-1)\pi}{n} \right),\sin \left(\frac{(\sigma(i)-1)\pi}{n} \right) \right)$, for any permutation $\sigma$ of $\{ 1,\ldots,n\}$.

Moreover, if also  $\sum_{j=1}^n u_jv_j=0$, then 
\[
 \sum_{1\ls i<j\ls n} |u_iv_j-u_jv_i| \gr1,
\]
and equality is attained when $u,v$ are distinct vectors from $\{\pm e_1,\ldots, \pm e_n \}$.
\end{proposition}

\subsection{A second proof through planar section of $B_1^n$}
We now present an alternative proof for the planar projections of the cube by employing planar sections of the cross-polytope $B_1^n$ and highlighting their connection to the Mahler volume of a convex body $K$, defined by $\mathcal{P}(K)=\operatorname{vol}(K)\operatorname{vol}(K^{\circ})$, where $K^{\circ}$ denotes the polar body of $K$. Notice that $\left( B_1^n \cap H \right)^{\circ} = 2\,\operatorname{Proj}_{H} Q_n .$
To proceed, we will require the following tools:

\begin{theorem}\label{mahler 2k-gon}\cite[Theorem 2]{meyer2011volume}
Let $P$ be a centrally symmetric polygon with $2k$ vertices. Then, $$\mathcal{P}(P)\ls \mathcal{P}(C_{2k})=4k^2\sin^2\left(\frac{\pi}{2k} \right).$$
 Equality is attained if and only if $P$ is an affine image of the regular $2k$-gon $C_{2k}$.
\end{theorem}

\begin{theorem}[Nazarov, in \cite{chasapis2022slicing}]\label{planar section cross}
Let  $n\gr3$. For any $2$-dimensional subspace $H$ of $\mathbb{R}^n$, we have 
\[
\operatorname{vol}_2(B_1^n\cap H)\gr\frac{n^2 \sin^3(\frac{\pi}{2n})}{\cos(\frac{\pi}{2n})}.
\]
Equality is attained when $B_1^n \cap H$ is a regular $2n$-gon.
\end{theorem}

\begin{proof}[Proof of Theorem \ref{planar proj cube Chak-Fill}]
We begin by noticing that $B_1^n\cap H$ is a convex symmetric $2k$-gon, $k\ls n$. Thus, from Theorem~\ref{mahler 2k-gon} we obtain 
\[
\operatorname{vol}_2(B_1^n\cap H)\operatorname{vol}_{2}(\operatorname{Proj}_{H}Q_n)\ls k^2\sin^2 \left(\frac{\pi}{2k} \right)\ls n^2\sin^2 \left(\frac{\pi}{2n}\right),
\]
since $x\mapsto\frac{\sin(x)}{x}$ is positive and decreasing  on $[0,\pi)$.
Thus, using also Theorem~\ref{planar section cross}
\[
\operatorname{vol}_{2}\left(\operatorname{Proj}_HQ_n \right)\ls \frac{n^2\sin^2 \left(\frac{\pi}{2n}\right)}{\operatorname{vol}_k(B_1^n\cap H)}\ls \cot\left(\frac{\pi}{2n}\right).
\]
Equality is obviously attained at the regular $2n$-gon.
\end{proof}

\subsection{Generalizations to $L_p$-projection bodies}
There also exists a well-established connection between the volume of hyperplane projections of a convex body $K \subset \mathbb{R}^n$ and its \emph{projection body} $\Pi K$. The latter is a symmetric convex body defined via its support function as
\[
h_{\Pi K}(a) := \frac{1}{2} \int_{\mathbb{S}^{n-1}} |\langle x, a \rangle| \, S_{n-1}(K, dx) 
= \operatorname{vol}_{n-1}(\operatorname{Proj}_{a^{\perp}} K).
\]
Finally, we consider generalizations of~\eqref{main thm},~\eqref{cauchy cube} and the respective cross-polytope formula (see \cite{nayar2023extremal})  in the context of the \emph{$L_p$-projection bodies} $\Pi_p K$, $p\gr1$  of a convex body $K$ in $\mathbb{R}^n$ with $0\in \text{int}(K)$ defined by 
\[
h_{\Pi_pK}^p(a)=\frac{1}{2}\int_{\mathbb{S}^{n-1}}|\langle x,a \rangle|^{p}dS_{p}(K,x).
\]
Notice that for $\Pi_1K=\Pi K$. 

Let also $\Delta_c^n:=\Delta_n-\text{bar}(\Delta_n)$ be the centered simplex. We prove the following formulas:
\begin{theorem}\label{Formulas for L_p proj}
Let $a \in \mathbb{S}^{n-1}$ then

\begin{equation}\label{L_p cube}
    h^{p}_{\Pi_pQ_n}(a)=\frac{1}{2^{1-p}}\sum_{j=1}^{n}|a_j|^p,
\end{equation}
\begin{equation}\label{L_p cross}
    h_{\Pi_p B_1^n}^p(a)=\frac{2^{n-1}}{(n-1)!} \E \left|\sum_{j=1}^na_j \varepsilon_j \right|^p,
\end{equation}
and for $\sum_{i=1}^{n+1}a_i=0$ and $\sum_{j=1}^{n+1}a_j^2=1$
\begin{equation}\label{L_p simp}
   h_{\Pi_p \Delta_c^n}^p(a)=\frac{(n+1)^{\frac{2p-1}{2}}}{2(n-1)!}\sum_{j=1}^{n+1}|a_j|^p.
\end{equation}
Notice that for $p=1$ we retrieve the already known formulas.
\end{theorem}

\subsection{Sharp estimates}
Sharp estimates for \eqref{L_p cross} reduce to the classical Khintchine inequality, for which the optimal bounds were established by Haagerup \cite{haagerup1981best}.

Let,
\[
F_p(a) := \sum_{j=1}^{n} |a_j|^p.
\]

In the case where only the normalization condition  $\sum_{j=1}^na_j^2=1$ holds, the ordering of the extremal bounds reverses when passing between the regimes $1<p<2$ and $p>2$. Note that $F_2(a)=1$.
\begin{proposition}\label{bounds unc}
 Let $p\gr1$ and  $a_1,\ldots,a_{n}$ reals such that $\sum_{j=1}^{n}a_j^2=1$.  
\begin{itemize}
    \item For $1< p<2$, $F_p(a)$ is minimized at $(1,0,\ldots,0)$ and maximized at $\left( \frac{1}{\sqrt{n}},\ldots,\frac{1}{\sqrt{n}}\right).$

    \item For $p>2$,  the roles of the extrema are reversed: 
the configuration that maximizes $F_p(a)$ for $1 < p < 2$ now yields the minimum, 
and conversely, the previous minimum becomes the maximum.
    \end{itemize}
\end{proposition}

Finding the extrema of $F_p$ under the constraints $\sum_{j=1}^na_j^2=1$ and $\sum_{j=1}^n a_j=0$ is of independent interest. In the recent article \cite{brazitikos2025sharp} (or see \cite{rivin2002counting} for the appearance in a combinatorial context)  the authors treated the problem for the closely related quantity, among other symmetric polynomials, 
\[
p_k(a)=\sum_{j=1}^na_j^k,
\]
i.e the power sum. They essentially proved (combining with a simple argument that reduces the extrema in triples of the form $\left(\underbrace{a,\ldots,a}_{\gamma_1},\underbrace{b,\ldots,b}_{\gamma_2},\underbrace{c,\ldots,c}_{\gamma_3} \right),$ where $\gamma_1a+\gamma_2b+\gamma_3c=0$, $\gamma_1a^2+\gamma_2b^2+\gamma_3c^2=1$ and $\gamma_1+\gamma_2+\gamma_3=n$) that the upper bound is achieved, for both $n$ even and odd, in the case where of all $a_i$ are  equal except one. When $n$ is even the lower bound, for $p_k$, occurs at the half plus-half minus vector as a direct application of the power mean inequality.

\section{Notation and background}\label{Notation and Background}
We endow $\mathbb{R}^n$ with the standard inner product $\langle x, y \rangle = \sum_{j=1}^n x_j y_j$ between two vectors $x = (x_1, \ldots, x_n)$ and $y = (y_1, \ldots, y_n)$ in $\mathbb{R}^n$ and denote by $\|x\|_2 = \sqrt{\langle x, x \rangle}$ the induced standard Euclidean norm. For $p\gr0$ and a vector $x=(x_1,\ldots,x_n)$ in $\mathbb{R}^n$ we define its $\ell_p$ norm by $\|x\|_p=\left(\sum_{j=1}^n|x_j|^p \right)^{1/p}$. The closed centered unit ball of $\mathbb{R}^n$ is denoted $B^n_2$ and for the unit sphere, we write $\mathbb{S}^{n-1} = \partial B^n_2$. 
Moreover, we write $e_1, \ldots, e_n$ for the standard basis vectors, $e_1 = (1, 0, \ldots, 0)$, $e_2 = (0, 1, 0, \ldots, 0)$ etc. As usual, for a set $A$ in $\mathbb{R}^n$, $A^\perp = \{ x \in \mathbb{R}^n : \langle x, a \rangle = 0 \ \forall a \in A \}$ is its orthogonal complement, with the convention that for a vector $u$ in $\mathbb{R}^n$, $u^\perp = \{u\}^\perp$ is the hyperplane perpendicular to $u$. The orthogonal projection onto an affine or linear subspace $H$ in $\mathbb{R}^n$ is denoted by $\text{Proj}_H$. Volume, i.e. $k$-dimensional Lebesgue measure in $\mathbb{R}^n$ is denoted by $\text{vol}_k(\cdot)$, identified with $k$-dimensional Hausdorff measure (normalised so that cubes with side-length 1 have volume 1).  Recall that a body in $\mathbb{R}^n$ is a compact set with nonempty interior.  If $0\in \text{int}(K)$, we define its polar as  $K^{\circ}= \left \{ y \in \mathbb{R}^n: \langle x,y \rangle  \ls 1 \,\,\, \text{for all $x \in K$}\right \}.$  Let $A\subset \mathbb{R}^n$ be non-empty and convex we define the support function $h_A:\mathbb{R}^n\rightarrow (-\infty,\infty]$ of $A$, non-empty and convex,  as $h_A(u)=\sup_{x\in A}  \langle x,u \rangle$ for $u\in \mathbb{R}^{n}$.

Let $\Delta_n$ denote the regular $n$-dimensional simplex of side length $\sqrt{2}$, which we view through its usual embedding in the hyperplane $\mathcal{H}:=\left \{ \sum_{i=1}^{n+1}x_i=1 \right\}$  in $\mathbb{R}^{n+1}$, namely
\[
\Delta_n = \left\{x\in\mathbb{R}^{n+1}: x_j\gr 0, \sum_{j=1}^{n+1}x_j=1\right\}.
\]
The set $\Delta_n$ has $n$-dimensional volume equal to $\frac{\sqrt{n+1}}{n!}$ and centroid at $\mathrm{bar}(\Delta_n):=\left(\frac{1}{n+1},\ldots,\frac{1}{n+1}\right)$. The normals to the facets  are given by permutation of $(-n,1\ldots,1)$. That is immediate since the centers of the facets are permutations of $\left(0,\frac{1}{n},\ldots,\frac{1}{n}\right)$ and by symmetry the vector pointing from$ \left(\frac{1}{n+1},\ldots,\frac{1}{n+1}\right)$ to $\left(0,\frac{1}{n},\ldots,\frac{1}{n}\right)$ is perpendicular to the facets.

\subsection{Majorization}
Given $x = (x_1, \ldots, x_n)$, we denote by $x^{*} = (x_1^{*}, \ldots, x_n^{*})$ its decreasing rearrangement, i.e.,
\[
x_1^{\ast} \ge x_2^{*} \ge \cdots \ge x_n^{\ast}.
\]

For any two vectors $x, y \in \mathbb{R}^n$, we say that $x$ is majorized by $y$, and write $x \prec y$, if
\[
\sum_{i=1}^n x_i = \sum_{i=1}^n y_i
\quad \text{and} \quad
\sum_{i=1}^k x_i^{*} \le \sum_{i=1}^k y_i^{*}
\quad \text{for every } k = 1, 2, \ldots, n-1.
\]

As a direct consequence, for every vector $a = (a_1, \ldots, a_n) \in \mathbb{R}^n_+$ such that $\sum_i a_i = 1$, we have
\begin{equation}\label{majorization seq}
\left( \frac{1}{n}, \ldots, \frac{1}{n} \right) \prec (a_1, \ldots, a_n) \prec (1, 0, \ldots, 0).
\end{equation}
A well established inequality in the theory of majorization is due to Karamata \cite{karamata1932inegalite}:
\begin{theorem}\label{Karamata}
Let $f:I\subset \mathbb{R}\rightarrow \mathbb{R}$ be a convex function on $I$. If $a_1,\ldots,a_n,b_1,\ldots,b_n\in I$   such that $(a_1,\ldots,a_n)\prec (b_1,\ldots,b_n)$, then 
\[
\sum_{j=1}^nf(a_j)\ls\sum_{j=1}^nf(b_j).
\]
\end{theorem}
For concise exposition on majorization and Schur-convexity, we refer for instance to Chapter II of \cite{bhatia2013matrix}.

\subsection{Intrinsic volume and Quermassintegrals}
Minkowski's theorem, and definition of the mixed volumes, states that if $K_1,\ldots ,K_m$ are non-empty, compact convex
subsets of ${\mathbb R}^n$, then the volume of $\lambda_1K_1+\cdots +\lambda_mK_m$ is a homogeneous polynomial of degree $n$ in
$\lambda_i>0$. One can write
\begin{equation*}\operatorname{vol}_n(\lambda_1K_1+\cdots +\lambda_mK_m)=\sum_{1\ls i_1,\ldots ,i_n\ls m}
V(K_{i_1},\ldots ,K_{i_n})\lambda_{i_1}\cdots \lambda_{i_n},
\end{equation*}
where the coefficients $V(K_{i_1},\ldots ,K_{i_n})$ are invariant under permutations of their arguments. The coefficient $V(K_{i_1},\ldots ,K_{i_n})$
is the mixed volume of $K_{i_1},\ldots ,K_{i_n}$. In particular, if $K$ and $D$ are two convex bodies in ${\mathbb R}^n$
then the function $\operatorname{vol}_n(K+\lambda D)$ is a polynomial in $\lambda\in [0,\infty )$:
\begin{equation*}\operatorname{vol}_n(K+\lambda D)=\sum_{j=0}^n \binom{n}{j} V_{n-j}(K,D)\;\lambda^j,\end{equation*}
where $V_{n-j}(K,D)= V(K[n-j],D[j])$ is the $j$-th mixed volume of $K$ and $D$ (we use  the notation $D[j]$ for $D,\ldots ,D$ $j$-times).
If $D=B_2^n$ then we set $W_j(K):=V_{n-j}(K,B_2^n)=V(K[n-j], B_2^n[ j])$; this is the $j$-th quermassintegral of $K$.
The intrinsic volumes $V_j(K)$ of $K$ are defined for $0\ls j\ls n$ by
\begin{equation}\label{eq:intrinsic}V_j(K)=\frac{\binom{n}{j}}{\omega_{n-j}}W_{n-j}(K).\end{equation}
 Let $K,K_1,\ldots,K_{n-1}$ be convex bodies in $\mathbb{R}^{n}$. The \emph{mixed surface area measure} $S(K_1,\ldots,K_{n-1},\cdot)$ is a Borel measure in $\mathbb{S}^{n-1}$ uniquely determined such that 
\[
V(K_1,\ldots,K_{n-1},K)=\frac{1}{n}\int_{\mathbb{S}^{n-1}}h_{K}(u) dS(K_1,\ldots,K_{n-1},u).
\]
We define $S_j(K,\cdot)=S(K[j],B_2^n[n-1-j])$ and when appropriate $S(K,\cdot):=S_{n-1}(K,\cdot)$. Surface area measures are a fundamental concept in the theory of convex bodies, i.e.
nonempty compact convex subsets of $\mathbb{R}^n$. Given  a convex polytope $P$ in $\mathbb{R}^n$  its surface area measure is given by 
\begin{equation}\label{S_{n-1}(P)}
S(P,\cdot):=S_{n-1}(P,\cdot)=\sum_{u\in \mathcal{N}(P)}\operatorname{vol}_{n-1}(F_u)\delta_u,
\end{equation}
where $\mathcal{N}(P)$ denotes the set of all unit facet normals of $P$ and $\delta_u$ is the Dirac probability measure supported at $u$.

\subsection{Projection bodies and $L_p$-projection bodies}

Let $K$ be a convex body in $\mathbb{R}^n$ the projection body $\Pi K$ of $K$ is a symmetric convex body in $\mathbb{R}^n$ defined through the support function 
\[
h_{\Pi K}(a):=\frac{1}{2}\int_{\mathbb{S}^{n-1}}|\langle a,x\rangle|dS_{n-1}(K,x)
\]

Surface area measures lie at the very core of the Brunn-Minkowski theory.  Based on Firey’s $L_p$ addition  for convex bodies, Lutwak  showed that the classical Brunn-Minkowski theory can be extended to an $L_p$ Brunn-Minkowski theory. We refer to \cite{schneider2013convex,gardner2006geometric} for more information on the $L_p$ theory.

Let $1\ls p<\infty$ and $K$ be a convex body in $\mathbb{R}^n$ Lutwak, Yang and Zhang in \cite{lutwak2000lp} introduced the $L_p$-projection body $\Pi_pK$, under a different normalization, defined as
\begin{equation}\label{L_p proj body}
h_{\Pi_p K}(a)=\frac{1}{2}\int_{\mathbb{S}^{n-1}}|\langle a,v\rangle|^p\,dS_{p}(K,dv),
\end{equation}
for $a\in \mathbb{S}^{n-1}$, where the $L_p$ surface measure $S_p(K,\cdot)$, the analog of of the surface area measure in the $L_p$ Brunn-Minkowski,  is defined via
\[
dS_{p}(K,\cdot)=h_{K}^{1-p}\,dS_{n-1}(K,\cdot).
\]

 For a convex polytope $P$ with $0\in \text{int}(P)$, the $L_p$ surface area measure $S_p(P,\cdot)$ is given by 
\begin{equation}\label{S_{p}(P)}
S_{p}(P,\cdot)=\sum_{u\in \mathcal{N}^{\ast}(P)}h^{1-p}_{P}(u)\operatorname{vol}_{n-1}(F_u)\delta_u,
\end{equation}
where $\mathcal{N}^{\ast}(P)$ denotes the set of unit facet normals of $P$ corresponding to facets which
do not contain the origin.

\subsection{Formulas for volume of projections}

\begin{theorem}[Cauchy-Minkowski]\label{C-M}
    Let $K$   be a convex body in $\mathbb{R}^{n}$. Then for every unit vector $a \in \mathbb{S}^{n-1}$, we have 
    \[
    \vol_{n-1}(\operatorname{Proj}_{a^{\perp}}K)=\frac{1}{2}\int_{\mathbb{S}^{n-1}} \left| \langle a,\xi\rangle \right| \,dS(K,\xi).
    \]
\end{theorem}
Let us illustrate intuitively this formula in the specific case of polytopes. Suppose we are given a convex polytope $P$ in $\mathbb{R}^n$ and we want to project it onto a hyperplane $a^\perp$, where $a$ is a unit vector. Let $\mathscr{F}_P$ be the set of facets of $P$. If $F \in \mathscr{F}_P$ then
\[
\text{vol}_{n-1}(\operatorname{Proj}_{a^\perp}(F)) = \text{vol}_{n-1}(F) \cdot |\langle a, n(F) \rangle|,
\]
where $n(F)$ is the unit outer-normal vector to $F$. Note that in $\operatorname{Proj}_{a^\perp}(P)$ every point is covered two times, so one gets the
following expression for the volume of projection
\begin{equation}\label{Cauchy-proj}
\text{vol}_{n-1}(\operatorname{Proj}_{a^\perp} P) = \frac{1}{2} \sum_{F \in \mathscr{F}_P} \text{vol}_{n-1}(F) \cdot |\langle a, n(F) \rangle|.
\end{equation}
 
The Cauchy–Minkowski formula~\ref{C-M} belongs to a broader family of integral representations known as \emph{Kubota’s formulas}.  
Kubota’s integral formula expresses the quermassintegrals of a convex body $K$ as averages of the volumes of its $(n-j)$-dimensional projections. For further background and detailed proofs, we refer the reader to \cite[Theorem~8.1.10 and (A.49)]{gardner2006geometric} or to \cite[Appendix~B]{ArtsteinAvidanGiannopoulosMilman2015}.

Let $F$ be a $m$-dimensional subspace of $\mathbb{R}^n$, for $1\ls m \ls n$. Then, for any convex body $K$ in $\mathbb{R}^n$ we have 
\begin{equation}\label{|U||P_EK|}
    \operatorname{vol}_{n-m}(L) \operatorname{vol}_m(\operatorname{Proj}_F K)=\binom{n}{m}V(K[m],L[n-m])
\end{equation}
where $L$ is any convex body in the subspace $F^{\perp}$. (See \cite[Theorem 5.3.1]{schneider2013convex})

Relation~\eqref{|U||P_EK|} now suggests that for any orthonormal system $u_1,\ldots,u_r$  
\begin{align*}\label{mixed vol formula for proj}
\operatorname{vol}_{n-r}(\operatorname{Proj}_{\text{span}^{\perp}\{u_1,\ldots,u_r\}}K)=\frac{n!}{(n-r)!}V(K[n-r],[0,u_1],\ldots,[0,u_r]),
\end{align*}
 where by $\text{span}\{ u_1,\ldots,u_n\}$ we denote the linear span of $u_1,\ldots,u_n$.

\section{Proofs}

As described earlier, we can assume that every affine subspace $H$ of $\mathcal{H}$ with relative codimension one passes through $\operatorname{bar}(\Delta_n)$, since the volume of projections is invariant under translations.  
Each such hyperplane $H$ extends to a codimension-one subspace of $\mathbb{R}^{n+1}$ by taking its affine hull of $H$ with the origin, yielding a hyperplane $a^\perp$ such that $\operatorname{bar}(\Delta_n)\in a^\perp$.  
Consequently, $a$ lies in $\mathcal{H}' = \Bigl\{ x \in \mathbb{R}^{n+1} : \sum_{i=1}^{n+1} x_i = 0 \Bigr\}.$

\begin{proof}[Proof of Theorem \ref{main thm}]
By  Cauchy's formula for polytopes \eqref{Cauchy-proj} we obtain
\begin{equation}\label{ncent_proj}
\operatorname{vol}_{n-1}\left(\operatorname{Proj}_{a^{\perp}\cap \mathcal{H}}\Delta_n \right)=\frac{1}{2}\frac{\sqrt{n}}{(n-1)!}\sum_{j=1}^{n+1} \left| \left\langle a+\frac{1}{n+1}\boldsymbol{1},\nu_i \right\rangle \right|,
\end{equation}
where $\boldsymbol{1}=(1,\ldots,1)$.
The unit normals  of  $\Delta_n$  are given by  $$\nu_i=\frac{1}{\sqrt{n(n+1)}}\left(1,\ldots,1,\underbrace{-n}_{i},1,\ldots,1\right)$$ for $i=1,\ldots,n+1$ and  since $\sum_{j=1}^{n+1} a_j = 0$, it follows that
\[
\sum_{j \neq i} a_j = -a_i,
\]
Therefore
\[
\left| \sum_{j \neq i} a_j - n a_i \right| = (n+1)|a_i|.
\]
This completes the proof.
\end{proof}

\begin{proof}[Proof of Proposition \ref{sum min}]
For $n = 3$, without loss of generality, assume that $a_3 \neq 0$ and set 
\[
x = \frac{a_1}{a_3}, \quad y = \frac{a_2}{a_3}.
\]
Then the desired inequality becomes
\[
\frac{|x| + |y| + 1}{\sqrt{x^2 + y^2 + 1}} \geq \sqrt{2},
\]
but since \( x + y + 1 = 0 \), this is equivalently written as
\[
\frac{|x| + |x + 1| + 1}{\sqrt{x^2 + x + 1}} \geq 2.
\]
This inequality is easily verified to be true.

We proceed by induction on $n$. Assume the statement holds for some $n$, we will prove that is also holds for $n+1$. Let $a_1, \ldots, a_{n+1} \in \mathbb{R}$ satisfy
\[
a_1 + \ldots + a_{n+1} = 0.
\]
By grouping the last two terms, observe that
\[
|a_1| + \ldots + |a_{n-1}| + |a_{n}| + |a_{n+1}| 
\geq |a_1| + \ldots + |a_{n-1}| + |a_{n} + a_{n+1}|.
\]
Applying the induction hypothesis to the $n$ terms $a_1, \ldots, a_{n-1}, a_{n} + a_{n+1}$ (which sum to zero), we obtain
\[
|a_1| + \ldots + |a_{n-1}| + |a_{n} + a_{n+1}| \geq \sqrt{2} \left( a_1^2 + \ldots + a_{n-1}^2 + (a_{n} + a_{n+1})^2 \right)^{1/2}.
\]
To guarantee that
\[
a_1^2 + \ldots + a_{n-1}^2 + (a_{n} + a_{n+1})^2 \geq a_1^2 + \ldots + a_{n-1}^2 + a_{n}^2 + a_{n+1}^2,
\]
we require $2a_{n}a_{n+1} \geq 0$, which holds if $a_{n}$ and $a_{n+1}$ are non-zero and have the same sign.

Therefore, by appropriately choosing $a_{n}$ and $a_{n+1}$ satisfying $a_{n}a_{n+1} \geq 0$, the inductive step is valid, completing the proof.
\end{proof}

\begin{proof}[Proof of Proposition \ref{sum max}]
When $n$ is even, applying the Cauchy--Schwarz inequality gives
\[
\sum_{j=1}^{n} |a_j| \leq \sqrt{n} \left( \sum_{j=1}^{n} a_j^2 \right)^{1/2} = \sqrt{n},
\]
since the sum of squares is given as $1$. Equality is attained when all components $|a_j|$ are equal and satisfy the constraints, which occurs at
\[
\left( \underbrace{\frac{1}{\sqrt{n}}, \ldots, \frac{1}{\sqrt{n}}}_{\frac{n}{2} \text{ times}}, \underbrace{-\frac{1}{\sqrt{n}}, \ldots, -\frac{1}{\sqrt{n}}}_{\frac{n}{2} \text{ times}} \right).
\]

For the odd case, without loss of generality, let $a_1, \ldots, a_k \gr 0$ and $a_{k+1}, \ldots, a_{n} \le 0$. Then, by two consecutive applications of the Cauchy--Schwarz inequality, one obtains

\begin{equation}\label{cs1}
    \sum_{i=1}^{k}a_i^2  \gr\frac{\left(\sum_{i=1}^ka_i \right)^2}{k} \quad \text{and}\quad 
     \sum_{i=k+1}^{n}a_i^2  \gr\frac{\left(\sum_{i=k+1}^{n} a_i \right)^2}{n-k}.
\end{equation}
Let $S := \sum_{i=1}^k a_i$. Summing relations~\eqref{cs1}, we obtain
\begin{equation*}\label{cs_conc}
S \ls \sqrt{\frac{k(n-k)}{n}},
\end{equation*}
since $\sum_{i=1}^{n} a_i = 0$.
 Then, 
\begin{equation}\label{concl}
    \sum_{i=1}^{n}|a_i|= 2S\ls2 \sqrt{\frac{k(n-k)}{n}},
\end{equation}
and the right hand side of \eqref{concl} is clearly maximized for $k=(n-1)/2$ or $k=(n+1)/2$. Equality is attained when $a_1=\ldots=a_{(n-1)/2}=\pm\sqrt{\frac{n+1}{n(n-1)}}$ and $a_{(n-1)/2+1}=\ldots=a_{n}=\mp\sqrt{\frac{n-1}{n(n+1)}}$.

\end{proof}

\begin{proof}[Proof of Theorem \ref{formula for planar cube}]
Since the volume of the projection is invariant under translations we can instead work with the shifted unit cube $[0,1]^n=\sum_{i=1}^n[0,e_i]$. From~\eqref{|U||P_EK|} and the multilinearity of the mixed volume we obtain,
\begin{align*}
\operatorname{vol}_{n-2}(\operatorname{Proj}_{\text{span}^{\perp}\{u,v\}}[0,1]^n)&=\frac{n!}{(n-2)!}V\left( [0,1]^n[n-2],[0,u],[0,v]\right)\\
&=\frac{n!}{(n-2)!}\sum_{i_1=1}^n\cdots\sum_{i_{n-2}=1}^nV \left([0,e_{i_1}],\ldots,[0,e_{i_{n-2}}],[0,u],[0,v]\right)                \\
&=\frac{1}{(n-2)!}\sum_{i_1=1}^n\cdots\sum_{i_{n-2}=1}^n|\det(e_{i_1},\ldots,e_{i_{n-2}},u,v)|            \\
&=\sum_{\substack{I \subseteq [n] \\|I|=n-2}}|\det(u,v,(e_i)_{i \in I})|\\
&=\sum_{1\ls i<j\ls n}|u_iv_j-u_j v_i|.
\end{align*}
\end{proof}

In order to prove Proposition~\ref{bound planar cube} we need the following Lemma.
\begin{lemma}\label{matrix lemma}\cite[Lemma 7]{EinollahzadehNematollahi+2025}
Let $d_1,\ldots,d_n\gr0$ and $\omega_1,\ldots,\omega_n\in \{z\in \mathbb{C}: |z|=1\}$. Then,
\[
\sum_{1\ls i,j \ls n}d_id_j|\omega_i-\omega_j|\ls \frac{2}{n}\cot\left(\frac{\pi}{2n}\right) \left(\sum_{i=1}^nd_i \right)^2.
\]
Equality holds if $\{\omega_1,\ldots,\omega_n\}$  is the set of all $n$-th roots of unity and $d_1=\cdots=d_n$.
\end{lemma}

\begin{proof}[Proof of Proposition \ref{bound planar cube}]
For the lower bound notice that Lagrange's identity suggests 
\begin{equation}\label{Lagrange identity}
\sum_{1\ls i<j\ls n}( u_iv_j-u_jv_i)^2=\|u\|_2^2\|v\|_2^2-\langle u,v \rangle ^2=1.
\end{equation}
Thus, it is immediate that
\[
 \sum_{1\ls i<j\ls n} |u_iv_j-u_jv_i| \gr \left( \sum_{1\ls i<j\ls n} |u_iv_j-u_jv_i|^2 \right)^{1/2}=1.
\]
Equality occurs, when exactly one of $|u_iv_j-u_jv_i|$ is non-zero, say $(p,q)$ with $p\neq q$, that is $|u_iv_j-u_jv_i|=0$ for $(i,j)\neq (p,q)$ and $|u_pv_q-u_qv_p|\neq 0$. The last condition says that $(u_p,v_p)$ cannot be collinear with $(u_q,v_q)$. But the first condition yields that $(u_i,v_i)$, for $i\neq p,q$, is colinear with both $(u_p,v_p)$ and $(u_q,v_q)$. This leads to $(u_i,v_i)=0$ for each $i \neq p,q$ and thus $u_p^2=1$ and $v_q^2=1$ or $v_q^2=1$ and $v_p^2=1$.

For the upper bound we use the following trigonometric parametrization: Since $(u_i,v_i)\in \mathbb{R}^2$  then $(u_i,v_i)=r_i(\cos(\theta_i),\sin(\theta_i))$ with $r_i\gr0$ and $0\ls \theta\ls 2\pi$.  Then, $u_i v_j - u_j v_i = r_i r_j \sin(\theta_j-\theta_i),$ so the quantity of interest becomes
\begin{equation}\label{trigonometric quant}
  \sum_{1\le i<j\le n} r_i r_j \big|\sin(\theta_j-\theta_i)\big|,
\end{equation}
where we also have $\sum_{i=1}^nr_i^2=2$ since $u,v$ are unit. Applying now Lemma \ref{matrix lemma} for $d_i=r_i$ and $\omega_i=e^{2i\theta_i}\in \mathbb{S}^1$ we obtain 
\begin{equation}
   \sum_{1\le i<j\le n} r_i r_j \big|\sin(\theta_j-\theta_i)\big|\ls\frac{1}{2n} \cot\left( \frac{\pi}{2n} \right)\left(\sum_{i=1}^nr_i \right)^2\ls \cot\left( \frac{\pi}{2n} \right),
\end{equation}
where we used Cauchy--Schwarz in the last step. 
\end{proof}

\begin{comment}
\begin{proof}[Proof of Proposition~\ref{claim cube-cross}]
Let $k=1,2,3,n-1$. The section $B_{\infty}^n\cap H$ is a convex polygon in $\mathbb{R}^2$, a symmetric convex body in $\mathbb{R}^3$ and a section of $B_{\infty}$ with a hyperplane respectively. Thus, from the discussion in Section~\ref{Mahler}:
\[
\operatorname{vol}_k( B_{\infty}^n\cap H)\operatorname{vol}_k(\operatorname{Proj}_HB_1^n)\gr \frac{4^k}{k!}.
\]
Thus,
\[
\operatorname{vol}_k(\operatorname{Proj}_HB_1^n)\gr\frac{4^k}{k!}\frac{1}{\operatorname{vol}_k( B_{\infty}^n\cap H)}\gr\frac{2^k}{k!}\frac{1}{C_{\text{cube}}(n,k)}=\frac{2^k}{k!}C_{\text{cross}(n,k)}.
    \]
\end{proof}
\end{comment}

\begin{proof}[Proof of Theorem \ref{Formulas for L_p proj}]
The $L_p$-measure of polytopes of~\eqref{S_{p}(P)} suggests that for a polytope $P$ in $\mathbb{R}^n$ with $0\in \text{int}(P)$ :
\begin{equation}\label{h_{piP} polytope}
 h_{\Pi_pP}^p(a)  =\frac{1}{2}\sum_{F\in \mathcal{F}_{n-1}}\operatorname{vol}_{n-1}(F)|\langle a,\nu_F \rangle|^p \,h_P^{1-p}(\nu_F),
\end{equation}
where $\nu_F$ is the unit normal to the facet $F$.

For the unit cube $Q_n = \left[-\tfrac{1}{2}, \tfrac{1}{2}\right]^n$, formula~\eqref{L_p cube} follows immediately, since all facets $F$ have unit $(n-1)$-dimensional volume, the outer unit normals to the facets are given by the $2n$ vectors $\{\pm e_j : 1 \le j \le n\}$, and the support function is $h_{Q_n}(u) = \tfrac{1}{2} \sum_{i=1}^n |u_i|$.

The cross-polytope $B_1^n = \operatorname{conv}\{\pm e_1, \ldots, \pm e_n\}$ has $2^n$ congruent simplicial facets, each of $(n-1)$-dimensional volume $\tfrac{\sqrt{n}}{(n-1)!}$, with outer unit normals 
$\tfrac{1}{\sqrt{n}}(\pm 1, \ldots, \pm 1)$, 
and support function 
$h_{B_1^n}(u) = \max_{i=1,\ldots,n} |u_i|$. 
Thus,

\[
h^p_{B_1^n}(a)=\frac{1}{2(n-1)!}\sum_{\varepsilon\in \{-1,1\}^n}|\langle a,\varepsilon\rangle|^p=\frac{2^{n-1}}{(n-1)!}\E \left|\sum_{j=1}^na_j \varepsilon_j \right|^p,
\]
where the expectation is over independent random signs $\varepsilon_j$, $\mathbb{P}(\varepsilon_j=\pm1)=
1/2$.

Finally, the centered simplex 
$\Delta_c^n := \Delta_n - \operatorname{bar}(\Delta_n) \subseteq \mathcal{H}' := \{\, \sum_{i=1}^{n+1} x_i = 0 \,\} \subset \mathbb{R}^{n+1}$ 
It is easy to check that the vertices  are permutations of $\frac{1}{n+1}(n,-1,\ldots,-1)$ and  that  permutations of $$\left( -\sqrt{\frac{n}{n+1}},\sqrt{\frac{1}{n(n+1)}},\ldots\sqrt{\frac{1}{n(n+1)}}\right)$$ are the outer unit normals to the facets of $\Delta_c^n$. The support function is given by $h_{\Delta_c^n}(u)=\max_{i=1,\ldots,n+1}\langle u,x_i \rangle$, where $x_i$ are the vertices, the formula~\eqref{L_p simp} follows similarly as in proof of Theorem~\ref{main thm}.
\end{proof}

\begin{proof}[Proof of Propostion \ref{bounds unc}]
The function $x \mapsto |x|^{p/2}$ is convex for $p \ge 2$ and $x \in \mathbb{R}$, and concave for $p \le 2$ and $x \ge 0$. The results follows immediately from Karamata's inequality~\ref{Karamata} since~\eqref{majorization seq} suggests $$\left(\frac{1}{n},\ldots,\frac{1}{n} \right)\prec (a_1^2,\ldots,a_n^2)\prec (1,0,\ldots,0).$$
\end{proof}

\noindent

{\bf Acknowledgements} The  author acknowledges support by the Hellenic Foundation for Research and Innovation (H.F.R.I.) under the call “Basic research Financing (Horizontal support of all Sciences)” under the National Recovery and Resilience Plan “Greece 2.0” funded by the European Union–NextGeneration EU (H.F.R.I. Project Number:15445). The author is much obliged to Silouanos Brazitikos for many fruitful discussions and for his help with the final presentation of this work, and is also grateful to Dimitris--Marios Liakopoulos for helpful communication.

\bibliographystyle{amsplain}
\bibliography{citations}

@book{ArtsteinAvidanGiannopoulosMilman2015,
  author    = {Shiri Artstein-Avidan and Apostolos Giannopoulos and Vitali D. Milman},
  title     = {Asymptotic Geometric Analysis, Vol. I},
  series    = {Mathematical Surveys and Monographs},
  volume    = {202},
  publisher = {American Mathematical Society},
  address   = {Providence, RI},
  year      = {2015},
  pages     = {xx+451},
  isbn      = {978-1-4704-1506-4},
  url       = {https://bookstore.ams.org/surv-202}
}

@article{alexander1977width,
  title        = {The width and diameter of a simplex},
  author       = {R. Alexander},
  journal      = {Geometriae Dedicata},
  volume       = {6},
  number       = {1},
  pages        = {87--94},
  year         = {1977},
  publisher    = {Springer},
  doi          = {10.1007/BF00181583},
  url          = {https://link.springer.com/article/10.1007/BF00181583}
}

@article{filliman1988largest,
  title        = {The largest projections of regular polytopes},
  author       = {P. Filliman},
  journal      = {Israel Journal of Mathematics},
  volume       = {64},
  number       = {2},
  pages        = {207--228},
  year         = {1988},
  publisher    = {Springer}
}

@article{filliman1990exterior,
  title        = {Exterior algebra and projections of polytopes},
  author       = {P. Filliman},
  journal      = {Discrete \& Computational Geometry},
  volume       = {5},
  number       = {3},
  pages        = {305--322},
  year         = {1990},
  publisher    = {Springer}
}

@article{filliman1992volume,
  title        = {The volume of duals and sections of polytopes},
  author       = {P. Filliman},
  journal      = {Mathematika},
  volume       = {39},
  number       = {1},
  pages        = {67--80},
  year         = {1992},
  publisher    = {London Mathematical Society}
}

@article{filliman1990extreme,
  title        = {The extreme projections of the regular simplex},
  author       = {P. Filliman},
  journal      = {Transactions of the American Mathematical Society},
  volume       = {317},
  number       = {2},
  pages        = {611--629},
  year         = {1990}
}

@book{gardner2006geometric,
  title        = {Geometric tomography},
  author       = {R. J. Gardner},
  number       = {58},
  year         = {2006},
  publisher    = {Cambridge University Press}
}

@article{gritzmann1992inner,
  title        = {Inner and outer j-radii of convex bodies in finite-dimensional normed spaces},
  author       = {P. Gritzmann and V. Klee},
  journal      = {Discrete \& Computational Geometry},
  volume       = {7},
  number       = {3},
  pages        = {255--280},
  year         = {1992},
  publisher    = {Springer}
}

@article{har2023width,
  title        = {On the width of the regular \(n\)-simplex},
  author       = {S. Har-Peled and E. W. Robson},
  journal      = {arXiv preprint arXiv:2301.02616},
  year         = {2023},
  url          = {https://arxiv.org/abs/2301.02616}
}

@article{martini1991convex,
  title        = {Convex polytopes whose projection bodies and difference sets are polars},
  author       = {H. Martini},
  journal      = {Discrete \& Computational Geometry},
  volume       = {6},
  number       = {1},
  pages        = {83--91},
  year         = {1991},
  publisher    = {Springer}
}

@article{martini1992quermasses,
  title        = {On quermasses of simplices},
  author       = {H. Martini and B. Weissbach},
  journal      = {Studia Scientiarum Mathematicarum Hungarica},
  volume       = {27},
  pages        = {213--221},
  year         = {1992}
}

@article{weissbach1984besten,
  title        = {Zur besten Beleuchtung konvexer Polyeder},
  author       = {B. Weissbach and H. Martini},
  journal      = {Beitr{\"a}ge zur Algebra und Geometrie = Contributions to Algebra and Geometry},
  volume       = {17},
  pages        = {151--168},
  year         = {1984}
}

@article{nayar2023extremal,
  title        = {Extremal sections and projections of certain convex bodies},
  author       = {P. Nayar and T. Tkocz},
  journal      = {Harmonic Analysis and Convexity},
  volume       = {9},
  pages        = {343},
  year         = {2023},
  publisher    = {Walter de Gruyter GmbH \& Co KG}
}

@book{webb1996central,
  title        = {Central slices of the regular simplex},
  author       = {S. P. Webb},
  year         = {1996},
  publisher    = {University of London, University College London (United Kingdom)}
}

@article{weissbach1988schranken,
  title        = {Schranken f{\"u}r die Dicke der Simplexe},
  author       = {B. Weissbach},
  journal      = {Beitr{\"a}ge zur Algebra und Geometrie = Contributions to Algebra and Geometry},
  volume       = {26},
  pages        = {5--12},
  year         = {1988}
}

@book{schneider2013convex,
  title={Convex bodies: the Brunn--Minkowski theory},
  author={Schneider, Rolf},
  volume={151},
  year={2013},
  publisher={Cambridge university press}
}

@article{chakerian1986measures,
  title={The measures of the projections of a cube},
  author={Chakerian, G Donald and Filliman, Paul},
  journal={Studia Sci. Math. Hungar},
  volume={21},
  number={1-2},
  pages={103--110},
  year={1986}
}

@article{EinollahzadehNematollahi+2025,
url = {https://doi.org/10.1515/spma-2025-0040},
title = {Minimum trace norm of real symmetric and Hermitian matrices with zero diagonal},
title = {},
author = {Mostafa Einollahzadeh and Mohammad Ali Nematollahi},
pages = {20250040},
volume = {13},
number = {1},
journal = {Special Matrices},
doi = {doi:10.1515/spma-2025-0040},
year = {2025},
lastchecked = {2025-11-14}
}

@article{Zhong,
author = {Zong, Chuanming},
year = {2005},
month = {04},
pages = {181-212},
title = {What is known about unit cubes},
volume = {42},
journal = {Bulletin of The American Mathematical Society - BULL AMER MATH SOC},
doi = {10.1090/S0273-0979-05-01050-5}
}

@article{chasapis2022slicing,
  title = {Slicing $\ell_p$-balls reloaded: stability, planar sections in $\ell_1$},
  author={Chasapis, Giorgos and Nayar, Piotr and Tkocz, Tomasz},
  journal={The Annals of Probability},
  volume={50},
  number={6},
  pages={2344--2372},
  year={2022},
  publisher={Institute of Mathematical Statistics}
}

@inproceedings{meyer2011volume,
  title={On the volume product of polygons},
  author={Meyer, Mathieu and Reisner, Shlomo},
  booktitle={Abhandlungen aus dem Mathematischen Seminar der Universit{\"a}t Hamburg},
  volume={81},
  number={1},
  pages={93--100},
  year={2011},
  organization={Springer}
}

@article{mcmullen1984volumes,
  title={Volumes of projections of unit cubes},
  author={McMullen, Peter},
  journal={Bulletin of the London Mathematical Society},
  volume={16},
  number={3},
  pages={278--280},
  year={1984},
  publisher={Oxford University Press}
}

@article{haagerup1981best,
  title={The best constants in the Khintchine inequality},
  author={Haagerup, Uffe},
  journal={Studia Mathematica},
  volume={70},
  number={3},
  pages={231--283},
  year={1981},
  publisher={Polska Akademia Nauk. Instytut Matematyczny PAN}
}

@article{brazitikos2025sharp,
  title={Sharp inequalities for symmetric polynomials, Hunter's conjecture, and moments of exponential random variables},
  author={Brazitikos, Silouanos and Pandis, Christos},
  journal={arXiv preprint arXiv:2512.12254},
  year={2025}
}

@article{karamata1932inegalite,
  title={Sur une in{\'e}galit{\'e} relative aux fonctions convexes},
  author={Karamata, Jovan},
  journal={Publications de l'Institut mathematique},
  volume={1},
  number={1},
  pages={145--147},
  year={1932},
  publisher={Matemati{\v{c}}ki institut SANU}
}

@book{bhatia2013matrix,
  title={Matrix analysis},
  author={Bhatia, Rajendra},
  volume={169},
  year={2013},
  publisher={Springer Science \& Business Media}
}

@article{lutwak2000lp,
  title={Lp affine isoperimetric inequalities},
  author={Lutwak, Erwin and Yang, Deane and Zhang, Gaoyong},
  journal={Journal of Differential Geometry},
  volume={56},
  number={1},
  pages={111--132},
  year={2000},
  publisher={Lehigh University}
}

@article{rivin2002counting,
  title={Counting cycles and finite dimensional Lp norms},
  author={Rivin, Igor},
  journal={Advances in applied mathematics},
  volume={29},
  number={4},
  pages={647--662},
  year={2002},
  publisher={Elsevier}
}

\end{document}